\documentclass[a4paper,12pt]{amsart}

\author{Giles Gardam and Daniel J.~Woodhouse}
\title{The geometry of one-relator groups satisfying a polynomial isoperimetric inequality}

\usepackage[T1]{fontenc}
\usepackage[utf8]{inputenc}
\usepackage{amsfonts}
\usepackage{amsmath}
\usepackage{amssymb}
\usepackage{amsthm}
\usepackage{fullpage}
\usepackage{hyperref}
\usepackage{mathpazo}
\usepackage{mathtools}
\usepackage{parskip}
\usepackage{tikz}

\usepackage{setspace}
\theoremstyle{plain}
\newtheorem{thm}{Theorem}

\newtheorem{cor}[thm]{Corollary}
\newtheorem{lem}{Lemma}

\theoremstyle{definition}
\newtheorem{qn}[lem]{Question}
\newtheorem{prop}[lem]{Proposition}
\newtheorem{conj}[lem]{Conjecture}
\newtheorem{rmk}[lem]{Remark}

\newcommand{\Z}{\mathbb{Z}}
\newcommand{\gp}[2]{\langle \, #1 \, | \, #2 \, \rangle}
\newcommand{\abs}[1]{\left\lvert#1\right\rvert}

\mathtoolsset{centercolon}

\address{Department of Mathematics \\ Technion \\ Haifa \\ Israel}
\email{gilesgar@technion.ac.il}
\email{woodhouse.da@technion.ac.il}
\keywords{One-relator groups, Dehn functions, automatic groups, CAT(0) spaces, CAT(0) cube complexes, Baumslag--Solitar subgroups}
\subjclass[2010]{20F65, (20F67, 20E06, 20F05)}
\thanks{The first author was supported by the Israel Science Foundation (grant 662/15). The second author was supported by the Israel Science Foundation (grant 1026/15).}

\begin{document}

\begin{abstract}
For every pair of positive integers $p > q$ we construct a one-relator group $R_{p,q}$ whose Dehn function is $\simeq n^{2 \alpha}$ where $\alpha = \log_2(2p / q)$.
The group $R_{p, q}$ has no subgroup isomorphic to a Baumslag--Solitar group $BS(m, n)$ with $m \neq \pm n$, but is not automatic, not CAT(0), and cannot act freely on a CAT(0) cube complex.
This answers a long-standing question on the automaticity of one-relator groups and gives counterexamples to a conjecture of Wise.
\end{abstract}

\maketitle

\section{Introduction}

A classical topic in combinatorial and geometric group theory is \emph{one-relator groups}, that is, groups that can be defined by a presentation with only one relator.
Magnus proved that a one-relator group has solvable word problem, but the algorithmic complexity of the word problem remains unknown.
A geometric measure of this complexity is given by the Dehn function (see \cite{bridson_geom_wp} for a survey).
The Dehn function of a one-relator group can grow very quickly: the group $\gp{a, t}{a^{(a^t)} = a^2}$ has Dehn function $\operatorname{tower}_2 (\log_2(n))$, which is not bounded by any finite tower of exponents, but its word problem is nonetheless solvable in polynomial time \cite{MUW}.
This is conjecturally the largest Dehn function of a one-relator group; Bernasconi proved a weaker uniform upper bound, namely the Ackermann function \cite{bernasconi}.

On the other hand, much less is known about the intricacies of the geometry of one-relator groups satisfying a polynomial isoperimetric inequality, that is, whose Dehn function is bounded by a polynomial.
All previously known examples are hyperbolic or more generally automatic (see \cite{ECHLPT} for background on automatic groups), and thus have linear or quadratic Dehn function.
For example, every one-relator group with torsion is hyperbolic, and Wise has proved they are virtually special \cite{wise_qhc}.

A standard obstruction to a group having desirable geometry is the presence of a subgroup isomorphic to the Baumslag--Solitar group $BS(m, n) = \gp{a, t}{t^{-1} a^m t = a^n}$ for some $m \neq \pm n$: this group has a distorted cyclic subgroup, and its Dehn function is exponential.
A distorted cyclic subgroup rules out being hyperbolic, acting properly cocompactly on a CAT(0) space, or acting freely on a CAT(0) cube complex (of possibly infinite dimension).
A torsion-free one-relator group has geometric dimension 2 \cite{lyndon} so by a theorem of Gersten such a Baumslag--Solitar subgroup gives an exponential lower bound on Dehn function \cite[Theorem C]{gersten_dehn}, further ruling out automaticity which would require a quadratic isoperimetric inequality.

It has been asked whether such Baumslag--Solitar subgroups are the only pathologies for one-relator groups:

\begin{qn}
    \label{qn:muw_bs}
    Is it true that one-relator groups with no subgroups isomorphic to $BS(m,n)$, for $m \neq \pm n$, are automatic?
\end{qn}

\begin{conj}[Wise, {\cite[1.9]{wise_tubular}}]
    \label{conj:wise}
    Every [torsion-free] one-relator group with no subgroup isomorphic to $BS(m, n)$, for $m \neq \pm n$, acts freely on a CAT(0) cube complex.
\end{conj}

Question~\ref{qn:muw_bs} was articulated when the theory of automatic groups was first developing \cite[Problem 11 ff.]{gersten_problems} and was posed more recently by Myasnikov--Ushakov--Won \cite[1.5]{MUW}.
If true, it would imply that all polynomial Dehn functions of one-relator groups are linear or quadratic.

The first-named author introduced in his thesis~\cite{gardam} the one-relator groups
\[
    R(m, n, k, l) := \gp{x, y, t}{x^m = y^n, \, t^{-1} x^k t = x^l y } \cong \gp{x, t}{x^m (x^{-l} t^{-1} x^k t)^{-n} }
\]
for $\abs{m}, \abs{n} \geq 2$, $k \neq 0$, $l \not \equiv 0 \mod m$.
He then proved that they have no distorted Baumslag--Solitar subgroups, and that they are CAT(0) precisely when $\abs{k} > \abs{l + \frac{m}{n}}$.
In this paper we consider a subfamily of these groups: define \[
    R_{p, q} := R(2, 2, 2q, 2p-1) \cong \gp{x, y, t}{x^2 = y^2, \, t^{-1} x^{2q} t = x^{2p-1} y}.
\]

\begin{thm}
    \label{thm:main}
    Let $p > q$ be positive integers.
    The one-relator group $R_{p, q}$ has Dehn function $\simeq n^{2 \alpha}$ where $\alpha = \log_2(2p/q)$.
    In particular, it has no subgroup isomorphic to a Baumslag--Solitar group $BS(m, n)$ with $m \neq \pm n$, but is not automatic and not CAT(0).
\end{thm}

This answers Question~\ref{qn:muw_bs} negatively.
The key observation is that $R_{p,q}$ is virtually a \emph{tubular group}.
A group is tubular if it splits as a finite graph of groups with $\Z^2$ vertex groups and $\Z$ edge groups.

\begin{proof}[Proof of Theorem~\ref{thm:main}]
    It is demonstrated in Theorem~\ref{thm:snowflake_subgroup} below that $R_{p,q}$ has an index two subgroup that is isomorphic to the Brady--Bridson snowflake (tubular) group $G_{p,q}$.
    These groups are discussed in full in Section~\ref{section:snowflake}, but the salient fact is that $G_{p,q}$ has Dehn function $\simeq n^{2 \alpha}$ where $\alpha = \log_2(2p/q) > 1$; as the Dehn function is invariant up to finite index subgroups the first part of the statement holds.
    In contast, automatic and CAT(0) groups have at most quadratic Dehn function.
    Since $R_{p,q}$ is of geometric dimension 2 we conclude from Gersten's theorem that there are no such Baumslag--Solitar subgroups as their presence would force at least exponential Dehn function.
\end{proof}

\begin{rmk}
    Jack Button has shown that for odd $q \geq 3$, the group $R_{1,q}$ is \emph{not} residually finite, giving the first examples of one-relator groups that are CAT(0) (by \cite[Theorem G]{gardam}) but not residually finite~\cite{Button}.
     The proof shows that $G_{1,q}$ is non-Hopfian, which implies that $G_{1,q}$ and thus $R_{1,q}$ are not equationally Noetherian, resolving \cite[Problem 1]{baumslag_problems}; in fact one can extend Button's surjective endomorphism of $G_{1,q}$ with non-trivial kernel to show that $R_{1,q}$ itself is non-Hopfian, resolving \cite[Problem 7]{baumslag_problems}.
     Button has informed us that he also has completely determined for which~$p$ and $q$ the group $R_{p,q}$ is residually finite.
\end{rmk}

In~\cite{wise_tubular}, Wise classified the tubular groups that act freely on CAT(0) cube complexes.
In Section~\ref{section:cubulation} we apply this classification to disprove Conjecture~\ref{conj:wise}.
The Dehn function is a quasi-isometry invariant, so there are infinitely many quasi-isometry types of counterexamples to Question~\ref{qn:muw_bs} and Conjecture~\ref{conj:wise}.

\section{Virtually snowflake groups}
\label{section:snowflake}

In~\cite{BB} Brady and Bridson studied the groups \[
    G_{p,q} := \gp{ a, b, s, t }{ [a, b], \, s^{-1} a^q s = a^p b, \, t^{-1} a^q t = a^p b^{-1} }
\]
defined for positive integers $p$ and $q$.
Due to the suggestive nature of their van Kampen diagrams, these are called \emph{snowflake groups}.
The main theorem of their paper states that for $p \geq q$, the Dehn function of $G_{p,q}$ is $\simeq n^{2 \alpha}$ where $\alpha = \log_2(2p/q)$.
This gives the Dehn function of $R_{p,q}$, via the~following:

\begin{thm} \label{thm:snowflake_subgroup}
    The snowflake group $G_{p,q}$ is an index $2$ subgroup of the one-relator group~$R_{p,q}$.
\end{thm}

\begin{proof}
    First, we re-write the presentation of $R$ to exploit the fact that $\gp{x, y}{x^2 = y^2}$ is the fundamental group of the Klein bottle: we map $x \mapsto a$ and $y \mapsto ab$ to get \[
        R_{p,q} \cong \gp{a,b,t}{ a^{-1} b a b, \, t^{-1} a^{2q} t = a^{2p} b}.
    \]
    Let $X$ be the graph of spaces for $R_{p,q}$ with a vertex space a Klein bottle and edge space a cylinder.
    We can assume that the attaching maps are geodesics in the Klein bottle as in Figure~\ref{fig:cover2}.
    Let $X' \rightarrow X$ be the index two regular cover corresponding to the map to $\Z/2$ defined by $a \mapsto 1$, $b \mapsto 0$ and $t \mapsto 0$, indicated in Figure~\ref{fig:cover2}; on the Klein bottle subspace this is just the oriented double cover.

    \begin{figure}
        \centering
        \includegraphics[width=0.55\textwidth]{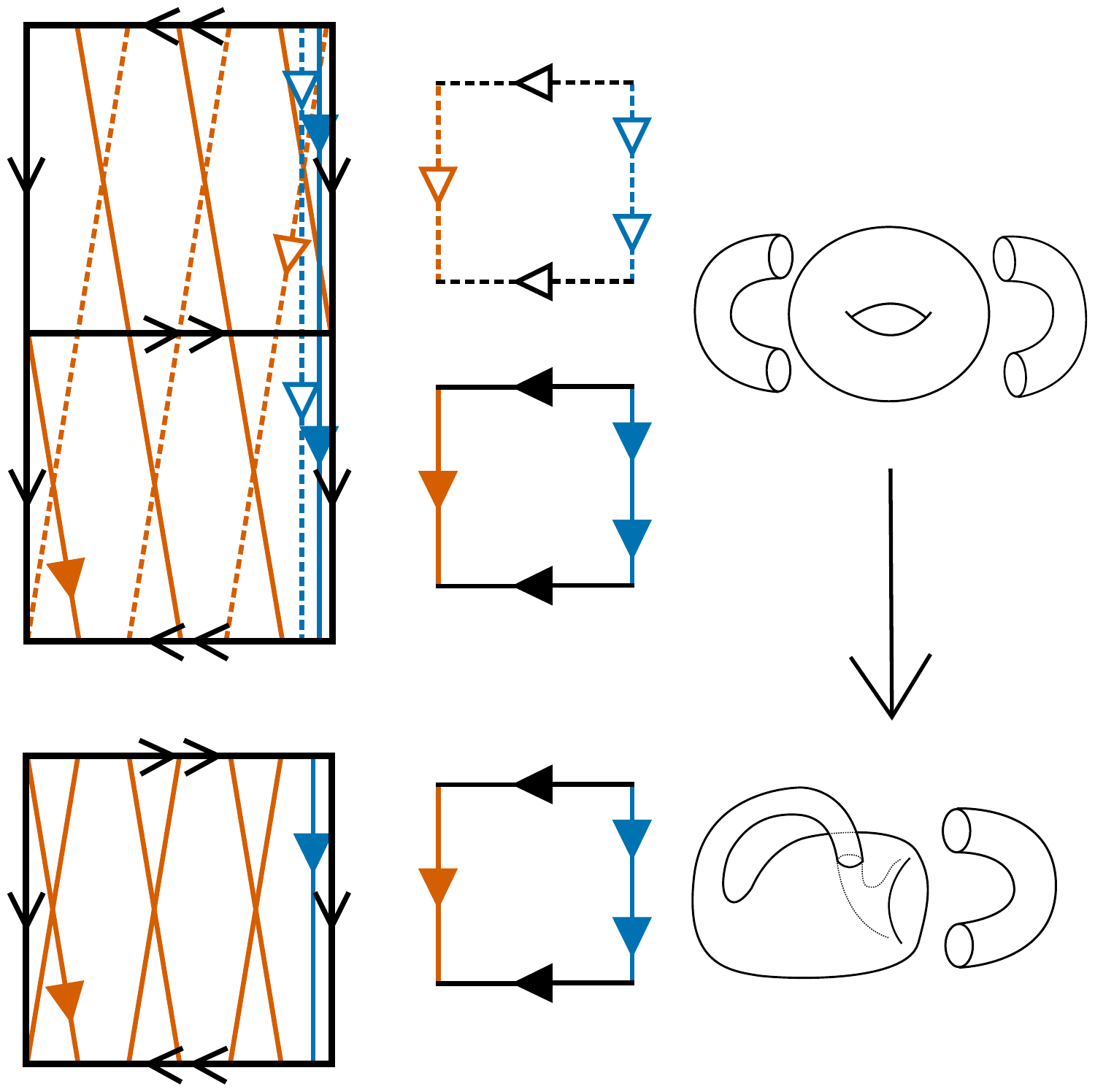}
        \caption{$G_{3,1}$ as an index $2$ subgroup of $R_{3,1}$. The two blue arrows in each cylinder are attached along $a$ and the orange arrows along geodesics representing (lifts of) $a^6 b$.}
        \label{fig:cover2}
    \end{figure}

    The fundamental group of $X'$ has the presentation \[
        \gp{x, y, s, t}{ [x, y], \, s^{-1} x^q s = x^p y, \, t^{-1} x^q t = x^{p} y^{-1}}
    \] where $x$ and $y$ are the generators of the fundamental group of the torus (corresponding to $a^2$ and $b$ respectively).
    This group is none other than $G_{p,q}$.
\end{proof}

\begin{cor}
    The set of exponents $\rho$ such that $n^\rho$ is the Dehn function of a one-relator group is dense in $[2, \infty)$.
\end{cor}

\begin{rmk}
    \cite[Problem 1.4]{MUW} asks whether quadratic Dehn function implies that a one-relator group is automatic.
    The snowflake group $G_{1,1}$ is Gersten's non-CAT(0) free-by-cyclic group introduced in \cite{gersten_cat}.
    It has been announced that this group is not automatic \cite{BR_free_by_cyclic}, which would settle this remaining problem as well.
\end{rmk}

\section{Non-cubulated examples}
\label{section:cubulation}

In~\cite{wise_tubular}, Wise gave a necessary and sufficient condition for a tubular group to act freely on a CAT(0) cube complex.
The condition is the existence of an \emph{equitable set} which permits the construction of immersed walls in the graph of spaces associated to the group.
A dual cube complex is then obtained from the corresponding wallspace.

\begin{prop}
    Let $p$ and $q$ be positive integers.
    The snowflake group $G_{p,q}$ acts freely on a CAT(0) cube complex if and only if $p \leq q$.
\end{prop}
\begin{proof}
    The existence of a free action of $G_{p,q}$ on a CAT(0) cube complex is equivalent to the existence of an equitable set: $S = \{ (u_1, v_1), \dots, (u_k, v_k) \} \subseteq \Z^2 \setminus \{ (0, 0) \}$ such that $[ \Z^2: \langle S \rangle ] < \infty$ and
        \[
        \sum_i \# [(q, 0), (u_i, v_i) ] = \sum_i \# [(p, 1), (u_i, v_i) ],
        \quad
        \sum_i \# [(q, 0), (u_i, v_i) ] = \sum_i \# [(p, -1), (u_i, v_i) ]
        \]
    where $\#[(a, b), (c, d)]$ denote the ``intersection number'' $\abs{ad - bc}$.
    Thus the problem reduces to solving
    \begin{equation}
        \label{eq:target}
        \tag{$\star$}
        \sum_i \abs{q v_i} = \sum_i \abs{p v_i - u_i} = \sum_i \abs{p v_i + u_i}.
    \end{equation}
    If $p \leq q$, then a solution is $\{ (q, 1), (q, -1) \}$.
    If $p > q$, there is no solution:
    \begin{align*}
    \sum_i \abs{p v_i - u_i} + \abs{p v_i + u_i} \geq \sum_i \abs{(p v_i - u_i) + (p v_i + u_i)} = 2 \sum_i \abs{p v_i} \geq 2 \sum_i \abs{q v_i}
    \end{align*}
    and equality can only hold in this last inequality if all $v_i = 0$, in which case some $u_i \neq 0$ and \eqref{eq:target} clearly cannot hold.
\end{proof}

\begin{cor}
    \label{cor:cube}
    Let $p > q$ be positive integers.
    Then the one-relator group $R_{p,q}$ has no subgroup isomorphic to $BS(m, n)$ for $m \neq \pm n$ but does not act freely on a CAT(0) cube complex.
\end{cor}

\begin{rmk}
    One can also deduce that for $p > q$ the group $G_{p,q}$ does not act freely on a CAT(0) cube complex from the fact that $G_{p,q}$ has a cyclic subgroup with distortion $n^\alpha$ \cite[Corollary 2.3]{BB} whereas cyclic subgroups are undistorted in groups admitting such actions by \cite[Theorem 1.5]{haglund_isometries} (which was generalized to finitely generated virtually abelian subgroups in \cite{woodhouse_axis}).
\end{rmk}

\subsection*{Acknowledgements}

We thank the referee for suggestions that improved the exposition of this paper.

\bibliography{geom-one-rel}{}
\bibliographystyle{alpha}

\end{document}